\documentclass[a4paper,12pt]{amsart}
\usepackage[top=1.5in, bottom=1in, left=0.9in, right=0.9in]{geometry}

\usepackage[utf8]{inputenc}
\usepackage[all]{xy}
\usepackage{color}
\usepackage{hyperref}
\usepackage{enumitem}

\usepackage{amsmath, mathtools}
\usepackage{hyperref}
\hypersetup{colorlinks=true}
\usepackage{amssymb}
\usepackage{amsfonts}
\usepackage{graphicx}
\usepackage{mathrsfs}
\usepackage{amscd}
\usepackage[all]{xy}
\usepackage{hyperref}
\usepackage{amsthm}
\usepackage{cleveref}
\usepackage{verbatim}
\usepackage{amscd}
\usepackage{mathtools}
\usepackage{comment}

\newtheorem{thm}{\bf Theorem}[section]

\newtheorem{prop}[thm]{\bf Proposition}

\newtheorem{cor}[thm]{\bf Corollary}
\theoremstyle{definition}
\newtheorem{definition}[thm]{\bf Definition}
\theoremstyle{remark}
\newtheorem{remark}[thm]{\bf Remark}

\newtheorem{question}[thm]{\bf Question}
\newtheorem{example}[thm]{\bf Example}
\numberwithin{equation}{section}

\DeclareMathOperator{\type}{{type}}
\DeclareMathOperator{\tr}{{trace}}

\DeclareMathOperator{\Ext}{{Ext}}
\DeclareMathOperator{\Hom}{{Hom}}

\DeclareMathOperator{\soc}{{{soc}}}

\DeclareMathOperator{\gr}{{{gr}}}

\DeclareMathOperator{\ord}{{ord}}

\DeclareMathOperator{\eli}{{eli}}

\DeclareMathOperator{\gll}{gll}
\DeclareMathOperator{\ulr}{{ulr}}
\DeclareMathOperator{\ind}{{index}}

\def\f0{\mathbf{0}}

\def\fn{\mathfrak{n}}

\def\fm{\mathfrak{m}}

\def \QQ{\mathbb Q}
\def \CC{\mathbb C}
\def \RR{\mathbb R}

\newcommand{\ses}[3]{0 \to {#1} \to {#2} \to {#3} \to 0}

\begin{document}

\title[Elias Ideals]{Elias Ideals}

\author[Hailong Dao]{Hailong Dao}
\address{Hailong Dao\\ Department of Mathematics \\ University of Kansas\\405 Snow Hall, 1460 Jayhawk Blvd.\\ Lawrence, KS 66045}
\email{hdao@ku.edu}

 \begin{abstract}
Let $(R, \mathfrak m)$ be a one dimensional local Cohen-Macaulay ring. An $\mathfrak m$-primary ideal $I$ of $R$ is Elias if the types of $I$ and of $R/I$ are equal. Canonical and principal ideals are Elias, and Elias ideals are closed under inclusion. We give multiple characterizations of Elias ideals and concrete criteria to identify them. We connect Elias ideals to other well-studied definitions: Ulrich, $\fm$-full, integrally closed, trace ideals, etc.  Applications are given regarding canonical ideals, conductors and the Auslander index. 
 \end{abstract}
 

\maketitle

\section*{Introduction}
Let $(R, \fm)$ be a local Cohen-Macaulay ring of dimension one and $I$ be an $\fm$-primary ideal of $R$. We say that $I$ is Elias if the Cohen-Macaulay types of $I$ and $R/I$ coincide. From standard facts, principal ideals or canonical ideals are Elias, and we will soon see that this property begets a rather rich and interesting theory.

Our work is heavily influenced by a nice result in \cite{Eli}, where Elias proves that any ideal $\omega$ that lies inside a high enough power of $\fm$ and such that $R/\omega$ is Gorenstein must be a canonical ideal. Although not stated explicitly there, the proof showed that any ideal that lies in a high enough power of $\fm$ is Elias, in our sense.  Another inspiration for the present work is \cite{DeS}, where De Stefani studies, in our language, powers of $\fm$ that are Elias in a Gorenstein local ring, and gives a  counter-example to a conjecture by Ding (see Section \ref{sec4} for the precise connection).  

In this note, we study Elias ideals in depth. They admit many different characterizations, and enjoy rather useful properties. For instance, they are {\bf closed under inclusion}, and principal or canonical ideals are Elias. On the other hand, conductor ideals or regular trace ideals are not Elias. When $R$ is Gorenstein, they are precisely ideals such that the Auslander index $\delta(R/I)$ is $1$. 

We are able to obtain many criteria to check whether an ideal is Elias, using very accessible information such as the minimal number or valuations of generators. Combining them immediately gives sharp bounds and  information on  conductor or canonical ideals, which can be tricky to obtain otherwise.  

There are several obvious ways to extend the present definitions and results to higher dimension rings or to modules. However, we choose to focus on the ideals in dimension one case here as they are already interesting enough, and also to keep the paper short.  We hope to address the more general theory in future works. 

We now describe briefly the structure and key results of the paper.

\begin{itemize}
\item In section \ref{sec1} we give the formal definition of Elias ideals and prove several key results. Theorem \ref{thmEli} contains several equivalent characterizations of Elias ideals. Corollary \ref{cor1} collects important consequences, for instance that Elias ideals are closed under ideal containment. Also, criteria for Elias ideals using colon ideals are given. Next, Proposition \ref{changerings} establishes the fundamental change of rings result that are used frequently in the sequence. 
\item Section \ref{sec2} connects Elias ideals to several well-studied class of ideals: Ulrich ideals, $\fm$-full ideals, full ideals, integrally closed ideals, etc. After some basic observations, (\ref{propUlrich}, \ref{prop2}, \ref{propmfull}), we give Theorems \ref{Elimu} and Proposition \ref{UElias}, which contain concrete ways to recognize Elias ideals using basic information such as number of generators or valuations. We also derive that conductor ideals or regular trace ideals are not Elias (Corollary \ref{cor_trace}). This indicates one of the useful application: if we know, for instance, that $\fm^2$ is Elias, then the conductor or any regular trace ideal must contains an element of $\fm$-adic order $1$. 

\item Given the previous section, it is natural to  study the Elias index $\eli(R)$, namely the first power of $\fm$ that is Elias, and we do so in Section \ref{sec3}. The first main result here is Theorem \ref{indexthm1}, connecting this index to the generalized L\"oewy length and the regularity of the associated graded ring. Next, in Theorem \ref{indexthm2}, we characterize rings with small indexes: $\eli(R)=1$ if and only if $R$ is regular, and $\eli(R)=2$ plus $R$ is Gorenstein is equivalent to $e(R)=2$. We give a large class of non-Gorenstein rings with Elias index $2$ (\ref{ind2_ex}). 

\item Lastly, in Section \ref{sec4} we focus on the special case of Gorenstein rings. In such situation, we observe that Elias ideals are precisely ones whose quotient has Auslander $\delta$-invariant one. This immediately allows us to apply what we have to recover old results about the Auslander invariant and Auslander index in  \ref{GorThm} and \ref{GorCor}. We give a counter-example to a Theorem by Ding and also revisit a counter-example to a conjecture by Ding given in \cite{DeS} (Examples \ref{Ding_ex} and \ref{DeS_ex}). 
\end{itemize}

\noindent\textbf{Acknowledgements}: It is a pleasure to thank Juan Elias and Alessandro Di Stefani for helpful comments and encouragements. The author is partially supported by the Simons Collaboration Grant FND0077558. 

\section{Elias ideals: definitions and basic results}\label{sec1} 

Throughout the paper, let $(R,\fm, k)$ be Cohen-Macaulay local ring of dimension one. For a module $M$, set $\type_R(M) = \dim_k\Ext^{\dim M}_R(k,M)$. Set $Q=Q(R)$  to be the total ring of fractions of $R$. Set $e=e(R)$, the Hilbert-Samuel multiplicity of $R$. For an element $x\in R$, the $\fm$-adic order of $x$, denoted $\ord(x)$ is the smallest $a$ such that $x\in \fm^a$. The order of an ideal $I$, denoted $\ord(I)$, is the minimum order of its elements.

\begin{definition}
We say that a $\fm$-primary ideal $I$ is an  Elias ideal if it satisfies  $\type(I) = \type(R/I)$. 
\end{definition}

\begin{thm}\label{thmEli}
We always have $\type(I)\geq \type(R/I)$. The following are equivalent. 

\begin{enumerate}
\item $\type(I) = \type(R/I)$.  
\item For any NZD $x \in \fm$, $xI:\fm\subseteq (x)$.
\item For any NZD $x \in \fm$, $xI:\fm = x(I:\fm)$.
\item For some NZD $x \in \fm$, $xI:\fm\subseteq (x)$.
\item  For some NZD $x \in \fm$, $xI:\fm = x(I:\fm)$.
\item $I:_Q\fm\subseteq R$. 
\item $K\subseteq \fm(K:_Q I)$ (assuming $R$ admits a canonical ideal $K$). 
\end{enumerate}

\end{thm}

\begin{proof}
Let $x$ be a NZD. Then $$\type(I) = \type(I/xI) = \dim_k\frac{xI:\fm}{xI} \geq \dim_k\frac{x(I:\fm)}{xI} =\dim_k\frac{I:\fm}{I}= \type(R/I)$$
Thus, $\type(I) = \type(R/I)$ if and only if $xI:\fm = x(I:\fm)$. Now, $xI:\fm \subseteq (x)$ is equivalent to  $xI:\fm =xJ$ for some ideal $J$, as $x$ is a NZD. Rewriting it as $xJ\fm\subseteq xI$, which is equivalent to $J\fm\subseteq I$, we get $J\subseteq I:\fm$. On the other hand $x(I:\fm)\subseteq xI:\fm$, thus $J=I:\fm$.   That establishes the equivalence of first five items. 

Note that for any NZD $x\in \fm$, $xI:\fm = x(I:_Q\fm)$. Thus, (6) is equivalent to (3). 

Let $K$ be a canonical ideal. Apply $\Hom_R(-, K)$ to the sequence $\ses{I}{R}{R/I}$, and indentifying $\Hom_R(I,K)$ with $K:_QI$, we get 
$\ses{K}{K:_Q I}{\Ext^1_R(R/I,K)= \omega_{R/I}}$. Since $\type(I) =\mu(K:_Q I)$ and $ \type(R/I) =\mu(\omega_{R/I})$, the equivalence of (7) and (1) follows.

\end{proof}

\begin{cor}\label{cor1}
We have:
\begin{enumerate}
\item If $I$ is isomorphic to $R$ or the canonical module of $R$ (assuming its existence), then $I$ is Elias.
\item If $I$ is Elias, then so is $J$ for any ideal $J\subseteq I$. (being  Elias is closed under inclusion) 
\item Let $K$ be a canonical ideal of $R$ and $I$ be an ideal containing $K$. Then $I$ is Elias if and only if $K\subseteq \fm(K:_RI)$. 
\item Let $K$ be a canonical ideal of $R$ and $I$ be an ideal such that $K\subseteq I$. Then $K:I$ is Elias  if and only if $K\subseteq \fm I$.
\item Suppose that $I$ contains a canonical ideal $K$ such that $\ord(K)=1$. Then $I$ is Elias if and only if $I=K$. 
\end{enumerate}
\end{cor}

\begin{proof}
For the first claim, $I:_Q\fm\subset I:_QI=R$. For the second claim, we have $J:_Q\fm\subset I:_Q\fm$. For (3), first note that $K:_QI \subset K:_QK=R$, so $K:_QI = K:_RI$, and we can use part (7) of Theorem \ref{thmEli}. 

For part (4), note that $K:(K:I)=I$ hence we can apply part (3). 

For part (5), we again apply part (3): if $K\subsetneq I$, then $\fm(K:_RI)\subseteq \fm^2$, contradicting $\ord(K)=1$. 

\end{proof}

The following change of rings result would be used frequently in what follows. 

\begin{prop}\label{changerings}
Let $(R,\fm)\to (S,\fn)$ be a local,  flat rings extension such that $\dim S=1$ and $S$ is Noetherian. Then $I$ is an Elias ideal of $R$ if and only if $IS$ is an Elias ideal of $S$.

\end{prop}

\begin{proof}
Under the assumption we have $\type_R(M)\type_{S/\fm S}(S/\fm S)= \type_S(M\otimes_R S)$ for any finitely generated $R$-module $M$ (see for instance \cite{FT}), thus the result follows. 
\end{proof}

\section{Elias ideals and other special ideals}\label{sec2}

\begin{definition}
Let $I$ be an $\fm$-primary ideal. 
\begin{itemize}
\item $I$ is called Ulrich (as an $R$-module) if  $\mu(I)=e(R)$. Assuming $k$ is infinite, then $I$ is Ulrich if and only if $xI=\fm I$ for some $x\in \fm$ (equivalently, for any $x\in \fm$ such that $\ell(R/xR)=e(R)$).  
\item $I$ is called $\fm$-full if $I\fm:x=I$ for some $x\in \fm$. 
\item $I$ is called full (or basically full) if $I\fm:\fm=I$. 
\end{itemize}
\end{definition}

\begin{remark}
When the definition of special ideals such as Ulrich or $\fm$-full ones involves an element $x$, we say that the property is witnessed by $x$. Note that being such $x$ is a Zariski-open condition (for the image of $x$ in the vector space $\fm/\fm^2$). For more on these ideals, see \cite{dms, HS, HRR, W}. 
\end{remark}

\begin{prop}\label{propUlrich}
Let $I$ be an $\fm$-primary ideal. Let $e$ be the Hilbert-Samuel multiplicity of $R$. The following are equivalent.
\begin{enumerate}
\item  $I$ is Ulrich.
\item  $\type(I)=e$. 
\end{enumerate}

\end{prop}
\begin{proof}
We can assume $k$ is infinite by making the flat extension $R\to R[t]_{(\fm,t)}$. Let $x\in \fm$ be such that $\ell(R/xR)=e$. Then $\ell(I/xI)=e$. Note that $\type(I)= \ell(\soc(I/xI)) \leq \ell(I/xI)=e$, and equality happens precisely when $\fm(I/xI)=0$, in other words, $I$ is Ulrich. 
\end{proof}

\begin{prop}\label{prop2}
Let $I$ be an $\fm$-primary ideal. 
\begin{enumerate}
\item Suppose $k$ is infinite. If $I$ is Ulrich, then it is $\fm$-full. 
\item Suppose $k$ is infinite. If $I$ is integrally closed, then it is $\fm$-full.  
\item If $I$ is $\fm$-full, then it is full. 
\end{enumerate}

\end{prop}

\begin{proof}
(1): We can find a NZD $x$ such that $Ix=I\fm$, so $I\fm:x = Ix:x=I$. \\
(2): see \cite[Theorem 2.4]{Goto}. \\
(3): We have $I\subseteq I\fm:\fm\subseteq I\fm:x$, from which the assertion is clear.

\end{proof}

\begin{prop}\label{propmfull}
If $I$ is $\fm$-full, witnessed by a NZD $x\in \fm$. The following are equivalent:
\begin{enumerate}
\item $I$ is Elias. 
\item $I=xJ$ for some  Ulrich ideal $J$. 
\end{enumerate}
\end{prop}

\begin{proof}
Assume $I$ is Elias, witnessed by a NZD $x$, so $I\fm:x=I$. We will show that $I\subseteq (x)$. If not, then $I$ contains an element $s$ whose image in $R/(x)$ is in the socle. Thus $s\fm\subset I\fm\cap(x) = x(I\fm:x)=xI$, so $s\in xI:\fm\subset (x)$, a contradiction. 

Since $I\subseteq (x)$ we must have $I=xJ$ for some $J$. We have $Jx=I= I\fm:x=Jx\fm:x=J\fm$, so $J$ is Ulrich. 

Assume (2). Then $I$ is Ulrich and also full by \ref{prop2}, so  $xI:\fm = \fm I:\fm=I=xJ\subset (x)$, thus $I$ is Elias. 
\end{proof}



\begin{cor}
If $e=2$ and $k$ is infinite, then $I$ is Elias if and only if $I\subseteq (x)$ for some NZD $x\in \fm$. 
\end{cor}

\begin{proof}
Since $e=2$, any ideal is either principal or Ulrich, and \ref{prop2} together with \ref{propmfull} give what we want. 
\end{proof}

\begin{thm}\label{Elimu} The following hold for an $\fm$ primary ideal $I$.
\begin{enumerate}
\item If $\mu(I)<e$ and $\type(R/I)\geq e-1$, then $I$ is Elias. 
\item Assume $\mu(\fm I)\leq \mu(I)=e-1$. Then $I\fm$ is Elias and $I\fm:\fm=I$. 
\item Furthermore, assume $R=S/(f)$ is a hypersurface, here $S$ is a regular local ring of dimension $2$. Let $J$ be an $S$ ideal minimally generated by $e$ elements, one of them is $f$. Then $JR$ is Elias. 
\end{enumerate}

\end{thm}

\begin{proof}
By the inequality $\type(I)\geq \type(R/I)$, we must have $\type(I)$ is $e$ or $e-1$. But if $\type(I)=e$, then $\mu(I)=e$ by \ref{propUlrich}, contradiction.

Next, we have: $$\type(R/I\fm)=\dim_k \frac{I\fm:\fm}{I\fm}\geq \dim_k \frac{I}{I\fm} = \mu(I)\geq e-1$$ and $I\fm$ is not Ulrich by assumption. So $I\fm$ is Elias and $\type(I\fm)=e-1$, which by the chain above implies that $I\fm:\fm=I$.

For the last  part, let $I=JR$. Then $\mu_R(I)=e-1$  and $\type(R/I)=\type(S/J)=e-1$, and we can apply the first part. 

\end{proof}

\begin{example}\label{exm2}
Let $R=k[[t^4,t^5,t^{11}]]\cong k[[a,b,c]]/(a^4-bc, b^3-ac, c^2-a^3b^2)$. Then $\fm^2$ is Elias: one can check directly or note that $\mu(\fm)=\mu(\fm^2)=3=e(R)-1$ and use \ref{Elimu}. But $\fm^2$ is not contained in $(x)$ for any $(x)$. 
\end{example}

\begin{example}
Let $R=k[[t^6,t^7,t^{15}]]\cong k[[a,b,c]]/(a^5-c^2, b^3-ac)$. Then the Hilbert function is $\{1,3,4,5,5,6,\dots\}$, thus $\fm^4$ is Elias. In this case, $\fm^4 \subseteq (a)$, so $\fm^4$ is trivially Elias.   \end{example}

Let $R\subset S$ be a finite birational extension. We recall that the conductor of $S$ in $R$, denoted $c_R(S)$, is $R:_{Q(R)}S$. 

\begin{prop}\label{prop_bir}
Let $R\subset S$ be a finite birational extension. If $IS=I$ (i.e, $I$ is an $S$-module) and $I$ is Elias, then $I:\fm\subseteq c_R(S)$. 
\end{prop}

\begin{proof}
Let $Q=Q(R)$. We have $R\supset I:_Q\fm$ = $IS:_Q \fm S \supset (I:\fm)S$, so $I:\fm\subseteq R:_QS =c_R(S)$ as desired.  
\end{proof}

Note that if $IS=I$, then $\tr(I)\subseteq c_R(S)$. So naturally, one can ask to extend \ref{prop_bir} as follows:

\begin{question}
If $I$ is Elias, do we have $I:\fm\subseteq \tr(I)$?\\

The answer is no. In Example \ref{exm2} above, Let $R=k[[t^4,t^5,t^{11}]]\cong k[[a,b,c]]/(a^4-bc, b^3-ac, c^2-a^3b^2)$. One can check that $\tr(\fm^2)=(a^2,ab,b^2,c)$ while $\fm^2:\fm=\fm$. 
\end{question}

\begin{cor}
Suppose $\fm^2$ is Elias (e.g., if $R$ has minimal multiplicity) and is integrally closed. If $\fm^2\subseteq c_R(\overline R)$  then $\fm\subseteq c_R(\overline R)$. 
\end{cor}

\begin{proof}
Apply \ref{prop_bir} to $I=\fm^2$. 
\end{proof}

\begin{cor}\label{cor_trace}
Assume that  the integral closure $\overline R$ is finite. Then the conductor  of $\overline R$ in $R$ is not Elias. A regular trace ideal is not Elias. 
\end{cor}
\begin{proof}
Let $\mathfrak{c} = c_R(\overline R)$. Then $\mathfrak{c}$ is a $\overline R$-module, so if it is Elias we would have $\mathfrak{c}:\fm\subseteq \mathfrak{c}$, absurd! Any regular trace ideal must contain $\mathfrak{c}$, see for instance \cite{dms}, so it can not be Elias either by \ref{cor1}. 
\end{proof}

The following is simple but quite useful for constructing Elias ideals from minimal generators of Ulrich ideals. See the examples that follow. 

\begin{prop}\label{UElias}
Let $I\subset J$ be regular ideals with $J$ Ulrich. Let $x\in \fm$ be a minimal reduction of $\fm$.  Assume that $\fm y\not\subseteq xI$ for any minimal generator of $J$. Then $I$ is Elias. 
\end{prop}

\begin{proof}
The assumption implies that $xI:\fm\subseteq \fm J = xJ\subset (x)$. 
\end{proof}

\begin{example}
Let $R=k[[a_1,\dots, a_n]]/(a_ia_j)_{1\leq i< j\leq n}$. Apply \ref{UElias} with  $J=\fm, x=a_1+a_2+\dots +a_n$. Note that each element  $f\in \fm$ has the form $f= \sum \alpha_ia_i^{s_i}$ where $\alpha_i$s are units or $0$. Then $a_if= \alpha_i a_i^{s_i+1}$ and $xf = \sum  \alpha_ia_i^{s_i+1}$. It follows easily then that the condition $\fm y\not\subseteq xI$ for any minimal generator $y$ of $\fm$ is equivalent to $a_i^2\notin xI$ for each $i$, which is equivalent to $a_i\notin I$ for each $i$. 

For instance, if $R=\QQ[[a,b,c]]/(ab,bc,ca)$, $I=(a-b,b-c)$ is Elias. Since $R/I = \QQ[[a]]/(a^2)$ is Gorenstein, $I$ is a canonical ideal. 
\end{example}
One can use valuations to construct Elias ideals from part of a minimal generating set of some Ulrich ideal. 

\begin{example}
Let $R=k[[t^n, t^{n+1},\dots, t^{2n-1}]]$. Let $I= (t^n,\dots, t^{2n-2})$. Apply  \ref{UElias} with  $J=\fm, x=t^n$. Let $\nu$ be the $t$-adic  valuation on $R$.   Note that for any minimal generator of $y\in J=\fm$, $3n-2\in \nu(y\fm)$. On the other hand $3n-2\notin \nu(xI)$, so $y\fm \not\subseteq xI$. It follows that $I$, and any ideal contained in $I$, is Elias. Note that again, since $R/I$ is Gorenstein, $I$ is actually a canonical ideal. 
\end{example}

\section{Elias index}\label{sec3}

\begin{definition} One defines the following:
\begin{itemize}
 \item Let the Elias index of $R$, denoted by $\eli(R)$ be the smallest $s$ such that $\fm^s$ is Elias. 
 \item Let the generalized L\"oewy length of $R$, denoted by $\gll(R)$, be the infimum of $s$ such that $\fm^s\subseteq (x)$ for some $x\in \fm$. 
 \item Let the Ulrich index  of $R$, denoted by $\ulr(R)$ be the smallest $s$ such that $\fm^s$ is Ulrich, that is $\mu(\fm^s)=e$. 
  \end{itemize}
\end{definition}

\begin{thm}\label{indexthm1} We have:
\begin{enumerate}
\item $\eli(R)\leq \gll(R)$. 
\item  $\gll(R)\leq \ulr(R)+1$, if the residue field $k$ is infinite.
\item Suppose that the associated graded ring $\gr_{\fm}(R)$ is Cohen-Macaulay and the residue field $k$ is infinite. Then $\eli(R) = \gll(R) = \ulr(R)+1$. 
\end{enumerate}
\end{thm}

\begin{proof}
If $\fm^s\subseteq (x)$ then $x$ must be a NZD. Thus $\fm^s$ is Elias by \ref{cor1}. The second statement follows from definition. The condition that $\gr_{\fm}(R)$ is Cohen-Macaulay implies that $\fm^s$ is $\fm$-full for all $s>0$, so the last assertion follows from \ref{propmfull}.
\end{proof}

\begin{thm}\label{indexthm2}
We have:
\begin{enumerate}
\item $\eli(R) =1$ if and only if $R$ is regular. 
\item Assume $R$ is Gorenstein, then $\eli(R)=2$ if and only if $e(R)=2$. 
\item Let $(A, \fn)$ be a Gorenstein local ring of dimension one. Suppose that $R=\fn:_{Q(A)}\fn$ is local. Then $\eli(R)\leq 2$. 
\end{enumerate}

\end{thm}

\begin{proof}
(1): Assume  $\fm$ is Elias. To show that $R$ is regular, we can make the extension $R\to R[t]_{(\fm,t)}$ and assume $k$ is infinite. Choose a NZD $x\in \fm-\fm^2$, we have $\fm^2:x=\fm$, that is $\fm$ is $\fm$-full witnessed by $x$. Then \ref{propmfull} shows that $\fm\subset (x)$, thus $\fm$ is principal. 

(2): We can assume again by \ref{changerings} that  $k$ is infinite. If $e=2$, then $\fm^2\subset (x)$ for a minimal reduction $x$ of $\fm$, thus $\fm^2$ is Elias. Now, suppose $\fm^2$ is Elias and $e\geq 3$, and we need a contradiction. We first claim that any Ulrich ideal $I$ of $R$ must lie in $\fm^2$. Take any  minimal reduction $x$ of $\fm$. Then $I\fm=xI\subseteq (x)$, so $I\subset (x):\fm\subseteq (x)+\fm^2$ (otherwise the socle of $R'= R/xR$ has order $1$, impossible as $R'$ is Gorenstein of length at least $3$). As $x$ is general, working inside the vector space $\fm/\fm^2$, we see that $I\subseteq \fm^2$.  

The set of $\fm$-primary Ulrich ideals in $R$ is not empty, as it contains high enough powers of $\fm$. Thus, we can pick an element $I$ in this set maximal with respect to inclusion. By the last claim, $I\subseteq \fm^2$, and hence $I$ is also Elias by \ref{cor1}. Now \ref{prop2} and \ref{propmfull} imply that $I=xJ$ for some NZD $x\in \fm$, so $J$ is an Ulrich ideal strictly containing $I$, and that's the contradiction we need.

(3): If $R=A$, then $\fn$ is Elias by \ref{thmEli}, hence $A$ is regular by part (1). Thus $R$ is also regular, and $\eli(R)=1$. If $R$ strictly contains $A$, then $c_A(R) = A:_{Q(A)}R= \fn$, hence $\fn \cong \Hom_A(R,A) \cong \omega_R$. So $\fn$ is a canonical ideal of $R$. On the other hand,  as $A$ is not regular, $\mu_A(R)=2$ (dualize the exact sequence $\ses{\fn}{A}{A/\fn}$ and identify $R$ with $\fn^*=\Hom_A(\fn, A)$). Thus $\ell_A(R/\fn)=2$, so $\ell_R(R/\fn)\leq 2$, which forces $\fm^2 \subset \fn$, and since $\fn$ is Elias, so is  $\fm^2$ by \ref{cor1}.  

\end{proof}

\begin{example}\label{ind2_ex}
We give some examples of item (3) in the previous Theorem. First let $A=\RR[[t,it]]$ with $i^2=-1$. Then $R =\CC[[t]]$. 

Next, let $H=\langle a_1,\dots, a_n \rangle$ be any symmetric semigroup and $b$ be the Frobenius number of $H$. Let $A=k[[H]]$ be the complete Gorenstein numerical semigroup ring 
of $H$. Then $R=k[[\langle a_1,\dots, a_n, b \rangle]]$ has Elias index $2$, unless if $H=\langle 2,3 \rangle$, in which case $\eli(R)=1$. 

Examples are $R=k[[ t^{e},t^{e+1}, t^{e^2-e-1}]]$ for $e\geq 3$. For such ring we have $\type(R)=2$, $e(R)=e$, $\gll(R)=e-1, \ulr(R)=e-1$, yet $\eli(R)=2$. These examples show that one can not hope to get upper bounds for $\gll(R)$ or $\ulr(R)$ just using $\eli(R)$. 
\end{example}

\section{Elias ideals in Gorenstein rings and Auslander index}\label{sec4} 

In this section we focus on Gorenstein rings. Throughout this section, let $(R, \fm, k)$ be a local Gorenstein ring of dimension one and $I\subset R$ an $\fm$-primary ideal. Recall that for a finitely generated module $M$, the Auslander $\delta$ invariant of $M$, $\delta(M)$ is defined to be the smallese number $s$ such that there is a surjection $R^s\oplus N\to M$. The first $s$ such that $\delta(R/\fm^s)=1$ is called the Auslander index of $R$, denoted $\ind(R)$. 

It turns out that Elias ideals are precisely those who quotient has Auslander invariant one. We collect here this fact and a few others. They are mostly known or can be deduced easily from results in previous sections, or both.  

\begin{prop}\label{GorThm} Let $(R, \fm, k)$ be a local Gorenstein ring of dimension one and $I\subset R$ an $\fm$-primary ideal. We have:
\begin{enumerate}
\item $\delta(R/I) = 1$ if and only if $I$ is Elias. 
\item Suppose $R$ is Gorenstein. Then $I$ is Elias if and only if for each NZD $x\in I$, $x\in \fm(x:I)$. 
\item Suppose $R$ is Gorenstein. For a NZD $x\in I$, $x:I$ is Elias if and only if $x\in \fm{I}$. In particular, if $x\in \fm^2$, then $x:\fm$ is Elias. 
\item $I$ is Elias if and only if  $1 \in \fm{I^{-1}}$, where $I^{-1}= R:_{Q}I$. If $I$ is Elias, then $I\subseteq \fm\tr(I)$.

\end{enumerate}

\end{prop}

\begin{proof}
Part (1) is a special case of a result  by Ding, \cite[Proposition 1.2]{Din} and our definition of Elias ideal. Part (2) and (3) are special cases of (3) and (4) of \ref{cor1}, as in that case $(x)$ is isomorphic to the canonical module.

Part (4) is \cite[2.4, 2.5]{Din}, and also follows easily from results above: the first assertion is just a rewriting of (2).  For the second assertion, it follows from the first that $I\subseteq \fm II^{-1} = \fm\tr(I)$. 
\end{proof}

There have been considerable interest in the following question:

\begin{question}
Given an ideal $I$ with $\delta(R/I)=1$, when can one say that $I\subset (x)$ for some NZD $x\in \fm$?
\end{question}

For instance, a conjecture of Ding asks whether $\ind(R)=\gll(R)$ always.  From our point of view, this is of course just a question about Elias ideals and Elias index. Thus, one immediately obtains the following.

\begin{cor}\label{GorCor}
 Let $(R, \fm, k)$ be a local Gorenstein ring of dimension one and $I\subset R$ an $\fm$-primary ideal. 
 \begin{enumerate}
 \item If $I$ contains a NZD $x$ of order $1$, then $I$ is Elias if and only if $I=(x)$.
 \item $\ind(R)= \eli(R)$. 
\item $\ind(R)=\gll(R)=\ulr(R)+1$ if $k$ is infinite and $\gr_{\fm}(R)$ is Cohen-Macaulay (this happens for instance if $R$ is standard graded or if $R$ is a hypersurface). 
 \end{enumerate}
 \end{cor}

\begin{proof}
For part (1), we apply  (5) of Corollary \ref{cor1}.
Part (2) is trivial from part (1) of \ref{GorThm}. Part (3) is \cite[Theorem 2.1]{Din0}, \cite[Corollary 2.11]{DeS}, and  is also a consequence of \ref{indexthm1}.

\end{proof}

\begin{example}\label{Ding_ex}(Counter-examples to a result by Ding) In this example, we construct examples of homogenous Elias ideals that are not inside principal ideals. 

Let $S=k[[x_1\dots, x_n]]$, and $J$ be a homogenous ideal such that $R=S/J$ is Gorenstein. Let $f\in S$ be an irreducible element of degree at least $2$ but  lower than the initial degree of $J$, and such that the image of $f$ in $R$ is a NZD. Then $I=fR:\fm$ is Elias by \ref{GorThm} but $I$ is not inside any principal ideal. For by the irreducibility of $f$, we must have $fR:\fm=(f)$, absurd. 

This class of examples contradicts Theorem 3.1 in \cite{Din}, which claims that for $I$ homogenous in a graded Gorenstein $R$, $\delta(R/I)=1$ (equivalently, $I$ is Elias) if and only if $I\subseteq (x)$ for some $x\in \fm$. 

For concrete examples, one can take $S=\QQ[[a,b]]$, $J=(a^3 -b^3)$,  and $f=a^2+b^2$. If one wants algebraically closed field, one can take $S=\CC[[a,b,c]]$, $J$ is a complete intersection of two general cubics, and $f=a^2+b^2+c^2$. 

The mistake in \cite[Theorem 3.1]{Din} is as follows. First, one derives that $1=\sum \frac{z_iy_i}{x_i}$ with $z_i\in \fm$ and $\frac{y_i}{x_i}\in I^{-1}$ and hence there is $i$ such that $\deg(z_iy_i)=\deg(x_i)$, which is correct. Then Ding claimed that there is $u\in k$ such that $z_iy_i=ux_i$. But this is not true. In the first example above we have $z_1=y_1=a, z_2=y_2=b, x_1=x_2=a^2+b^2$.  
\end{example}

\begin{example}\label{DeS_ex}(De Stefani's counter-example to a conjecture of Ding, revisited) As mentioned above, Ding conjectured that $\ind(R)=\gll(R)$ always when $R$ is Gorenstein. De Stefani gives a clever counter-example in \cite{DeS}. Let $S=k[x,y,z]_{(x,y,z)}$, $I=(x^2-y^5, xy^2+yz^3-z^5)$. Then $\ind(R)=5$ but $\gll(R)=6$. We now show how some parts of the proof in  \cite{DeS}, which is quite involved, can be shortened using our results. 

We note that since the Hilbert functions of $R$ are $(1,3,5,6,7,7,8,8...)$ and $e(R)=8$, we get that $\fm^5$ is Elias  by Theorem \ref{Elimu}. To conclude we need to show that $\fm^5$ is not contained in $(y)$ for any NZD $y\in \fm$. Note that $\fm^{6}$ is Ulrich by Hilbert functions. We first show one can assume $\ord(y)=1$. Assume  $\fm^{5}\subset (y)$,  $\fm^{5}=yI$, then  
$\fm^{5}\cong I$. If $\ord(y)\geq 2$, then $y\fm^{3}\subset \fm^5=yI$, so $ \fm^{3}\subset I$. But as $\fm I\cong \fm^{6}$ is Ulrich, we get $\fm^2I\subset (x)$ for some minimal reduction of $\fm$, thus $\fm^5 \subset \fm^2I\subset (x)$. For the rest, one can follow \cite{DeS}. 

\end{example}

\end{document}